\documentclass{article}
\usepackage{amsmath}
\usepackage{amsthm}
\usepackage{amsfonts}
\usepackage{amsbsy}
\usepackage{amssymb}
\usepackage[alphabetic]{amsrefs}
\usepackage{authblk}
\theoremstyle{plain}
  \newtheorem{theorem}{Theorem}
  \newtheorem{lemma}{Lemma}
  \newtheorem{corollary}{Corollary}

\theoremstyle{remark}

\theoremstyle{definition}
  \newtheorem{definition}{Definition}

\newcommand{\R}{{\mathbb R}}
\newcommand{\F}{{\mathbb F}}
\newcommand{\C}{{\mathbb C}}

\usepackage[pdf]{pstricks}

\title{The Fuglede Conjecture holds in \(\F_p^3\) for \(p=5,7\)}
\date{\today}

\author[$\dagger$]{Thomas Fallon} 
\author[$\star$]{Azita Mayeli} 
 \author[$\ddag$]{Dominick Villano}
 
 \affil[$\dagger$$\star$]{Department of Mathematics, The City University of New York \\ The Graduate Center, New York, NY  10016, USA} 
 
\affil[$\ddag$]{Department of Mathematics, University of Pennsylvania, 209 South 33rd Street
Philadelphia, PA 19104-6395, USA}

\begin{document}

\title{The Fuglede Conjecture holds in \(\F_p^3\) for \(p=5,7\)}



\date{}
\maketitle


\begin{abstract}
For \(p=5,7\), we show that a subset \(E \subset \F_p^3\) is spectral if and only if \(E\) tiles \(\F_p^3\) by translation.  Additionally, we give an alternate proof that the conjecture holds for \(p=3\).
\end{abstract}

\maketitle

\section{Introduction}

Let \(p\) be a prime number, \(\F_p\) the finite field with \(p\) elements, and \(\F_p^d\) the \(d\)-dimensional vector space over \(\F_p\).

\begin{definition}
A subset \(E \subset \F_p^d\) is called \emph{spectral} if there exists a subset \(A \subset \F_p^d\) such that
\[
\left\{  \chi_a \left(x \right) := e^{\frac{2\pi i }{p} x \cdot a } \mid a \in A\right\}
\]
forms an orthogonal basis of the complex vector space \(L^2 \left( E \right)\).  We then say \(A\) is a \emph{spectrum} of \(E\) and that \(\left(E, A \right)\) is a \emph{spectral pair}.
\end{definition}

\begin{definition}
A subset \(E \subset \F_p^d\) is said to \emph{tile \(\F_p^d\) by translations}, and simply called a \emph{tile}, if there exists a subset \(A \subset \F_p^d\) such that
\[
\sum_{a \in A} E \left( x-a \right) = 1
\]
for all \(x \in \F_p^d\).  Here, \(E \left( \cdot \right) \) denotes the indicator function of the set \(E\).  We call \(A\) a \emph{tiling set} of \(E\) and \(\left( E, A \right) \) a  \emph{tiling pair}.
\end{definition}

The  interest in the connection between spectral sets and tiles arises from a conjecture of Fuglede in \cite{FugMain}, which states that for a subset  \(\Omega \subset \R^d\) with  positive and  finite Lebesgue measure, \(L^2 \left( \Omega \right)\) has an orthonormal basis of exponentials  if and only if \(\Omega\) tiles \(\R^d \) by translations.  In \cite{TaoFug}, Tao disproved the Fuglede conjecture in \(\R^d\) for dimensions \(d \geq5 \) by lifting a non-tiling spectral set in \(\F_3^5\) to Euclidean space.  Examples of non-spectral tiles were found by Kolountzakis and Matolcsi in \(\R^d\), \(d \geq 5\), \cite{KMTilenospectra}.  Fuglede's conjecture was disproved in \(\R^4\) and \(\R^3\) in both directions, too; see \cite{MFugdim4, FRSFourTiles, BMPUni, MKComplexHad}.  However, the conjecture remains open in \(\R^2\) and \(\R\).  For partial results in special cases, see e.g.  \cite{IKTConvexFug, LabaTwoIntFug, LabaRootsTile}.

The methods of \cite{TaoFug} sparked interest in the discrete setting.  The examples of Tao from the previous paragraph show that the Fuglede conjecture is false in \(\F_p^d\) for \(d \geq 5\).   In \cite{IosevichCrew}, it is shown that the conjecture is   false for $\Bbb F_p^5$ for all odd primes $p$ and for  \(\F_p^4\) for all odd primes $p$ but the conjecture is true for $\F_2^4$ and  \(\F_p^3\)  when \(p=2,3\)  \cite{IosevichCrew,SamNatFug, SamMatt}.  In  \cite{IMPFug}, Iosevich, the second listed author  and Pakianathan proved that the Fuglede conjecture holds in \(\F_p^2\) for all primes $p$ by proving an equidistribution theorem that is conceptually analogous to arguments involving roots of unity in \cite{LabaTwoIntFug, LabaRootsTile}.   For prime $p$,  the conjecture is trivially true in \(\F_p\). Therefore,  the only remaining case is \(\F_p^3\), for $p>3$ and prime.  The main contribution of this paper  is the following result.

\begin{theorem} \label{thrm:main}
The Fuglede Conjecture holds in \(\F_5^3\) and \(\F_7^3\).
\end{theorem}

The proof of Theorem \ref{thrm:main} also yields an alternative  proof for the \(p=3\) case, which was originally covered in \cite{IosevichCrew}.   A computer-assisted proof for Fuglede\rq{}s conjecture in \(\F_5^3\) is recently presented in  \cite{PBirkFug}.

The main strategy of the proof is to examine the behavior of a prospective spectral set on one and two dimensional subspaces.  Namely, no large proportion of a spectral set's points can be contained in any given (affine) line or plane.  These estimates lead to combinatorial problems that, in general, seem to be difficult, but can be solved when \(p\) is very small.  This is one of the main difficulties in extending our methods beyond \(p=7\) and even represents the major technical improvement required to move from \(p=5\) to \(p=7\).

The paper is organized as follows: In Section \ref{sectionbackground}, we fix notation and introduce relevant previous results.    In Section \ref{sectionconcentrate} we prove various auxiliary lemmas that apply to any prime \(p\) and provide a short proof for the Fuglede Conjecture in \(\F_3^3\).  In Section \ref{sectionplancherel} we give a sufficient condition for a set \(E\) to be contained in the union of \(k\) parallel planes.  We then give a lemma concerning spectral sets that are contained in such a union, which again holds for any prime \(p\).  Then, as a corollary, we prove the Fuglede conjecture in \(\F_5^3\). 

In Section \ref{sectionseven}, we focus only on the case \(p=7\).  We first give Lemma \ref{Lm1} stating that a set with five points in \(\F_7^2\) with no three colinear points must determine at least 6 directions.  As corollaries, we show there are no spectral sets of size \(28\) or \(14\). In subsection \ref{section21notspectral}, Lemma \ref{Lm2} states that a set with seven points in \(\F_7^2\) with no four colinear points must  determine at least 6 directions.  Lemma \ref{twoparalellines} shows that  if there were a spectral set of size 21 that determines no more than two directions of two distinct planes, it must be contained in the union of seven parallel lines.  Lemma \ref{twentyoneproj} concerns equidistribution of functions \(f \colon \F_7^2 \to \{0,1,2,3\}\). This leads to Corollary \ref{21} proving that there are no spectral sets of size 21, completing the proof of the Fuglede Conjecture in \(\F_7^3\).

Applying the techniques of this paper to a few more primes seems possible.  A straightforward approach would entail more case-by-case analysis.  However, to handle all primes at once more robust geometric estimates would probably have to be established.  For example, the non-concetration lemmas described above rule out only a proportion of configurations that goes to zero as \(p\) grows.

\section*{Acknowledgements}

The authors would like to thank Philip Gressman for many helpful discussions. 

\section{Preliminaries and Notation} \label{sectionbackground}

A line in  \(\F_p^3\) will be any translate of a one-dimensional subspace of \(\F_p^3\), and a  plane will be any translate of a two-dimensional subspace of \(\F_p^3\).

These definitions hold in every dimension, and we are particularly interested in the case where the dimension is 3.  For any subspace \(S\), its orthogonal subspace will be denoted by \(S^\perp\).  In symbols:
\[
S^\perp:=\left\{ x \mid x \cdot y= 0 \ \text{for all} \ y \in S \right\}.
\]
As in the Euclidean case, \(\dim(S) + \dim(S^\perp)\) is the dimension of the whole space.
A direction will mean any one-dimensional subspace and be denoted by \(d\).  The family of planes orthogonal to \(d\) refers to the \(p\) distinct translates of \(d^\perp\).  If \( y \in d\) is nonzero, then we abuse notation and use \(y^\perp\) and  \(d^\perp\) interchangeably.
If \(E \subset \F_p^3\) and \(e_1-e_2 \in d\) for two distinct \(e_1, e_2 \in E\), we say \(E\) determines \(d\).  For a more general notion of determining directions, applicable beyond one dimensional subspaces, see \cite{IMPFug}.  Consideration of direction sets over finite fields originates in \cite{IMPDir}.  Counting directions, we have that in \(\F_p^d\) there are \(\frac{p^d-1}{p-1}\) directions, and in particular, a line has a single direction and a plane has \(p+1\) directions.

The Fourier transform of a function \(f \colon \F_p^d \to \C\) is defined as
\[
	\hat{f}(\xi) := p^{-d} \sum_{x \in \F_p^d} f(x) e^{-\frac{2 \pi i}{p} x \cdot \xi}, \forall \xi \in \F_p^d
\]
where the normalization is chosen so that the reconstruction formula is
\[
	f(x) = \sum_{\xi \in \F_p^d} \hat{f}(\xi)  e^{\frac{2 \pi i}{p} x \cdot \xi}
\]
and Plancherel's identity is 
\begin{align} \label{eqn:plancherel}
	\sum_{x \in \F_p^d} |f(x)|^2 = p^d \sum_{\xi \in \F_p^d} |\hat{f}(\xi)|^2.
\end{align}

In particular, for a set \(E \subset \F_p^d\), its indicator function will also be denoted by \(E\) and the Fourier transform of its indicator function is denoted by \(\hat{E}\) and is defined as
\[
\hat{E} \left( \xi \right) := p^{-d} \sum_{x\in E} e^{-\frac{2 \pi i }{p} x \cdot \xi} .
\]

The following result holds in all dimensions.  The language has been changed in order to fit the terminology of the previous paragraphs.

\begin{theorem} \label{equidist} \cite{HIPSVPack}
Suppose \(E \subset \F_p^3\) and \(\xi \in \F_p^3\), \(\xi \neq 0\).  Then \(\hat{E} \left( \xi \right) = 0\) if and only if \(E\) is equidistributed on the planes orthogonal to \(\xi\), meaning \(\left| E \cap \tau_x \left(\xi^\perp \right) \right|\) is constant as a function of \(x\in \F_p^3\).  Here \(\tau_x \left( \xi^\perp \right)\) is translation of \(\xi^\perp\) by \(x\).
\end{theorem}

The next theorem is a characterization of the size of spectral sets, due to Iosevich, the second listed author, and Pakianathan,  also holds in all dimensions.  It is an important ingredient for their proof of the Fuglede Conjecture in \(\F_p^2\).

\begin{theorem} \label{IMPdimfree} \cite{IMPFug} 
Let 
  \(E \subset \F_p^d\) and $p$ be an odd prime. Then \(A\) is a spectrum of \(E\) if and only if \(\left| A \right| = \left| E \right|\) and \(\hat{E} \left(a_1-a_2 \right) = 0 \) for any two distinct \(a_1, a_2 \in A\).  If \(E\) is spectral and \(\left| E \right| > p^{d-1}\), then \(E=A=\F_p^d\).  If \(E\) is spectral,  then \(\left| E \right| = 1\) or \(\left| E \right| = mp\)  where  \(1\leq m\leq p^{d-2}\).
\end{theorem}

The consequences of Theorems \ref{equidist} and \ref{IMPdimfree} in three dimensions are   summarized by the following result from \cite{IosevichCrew}.  In fact, this equivalence is one of several in the paper, but will be enough for present purposes.

\begin{theorem}\cite{IosevichCrew} \label{prop:easyred} Let $p$ be an odd prime.  Then the  following are equivalent: 
\begin{enumerate}
	\item The Fuglede Conjecture holds in \(\F_p^3\).
	\item There exist no spectral sets \(E\) satisfying \(\left| E \right| = mp\) with \(2 \leq m \leq p-1\).
\end{enumerate}
\end{theorem}

According to Theorem \ref{prop:easyred}, in order to prove the conjecture in \(\F_5^3\), it suffices to verify that no set of size \(10, 15\) or \(20\) is spectral.  Similarly, for \(\F_7^3\), it is enough to show that no set of size \(14, 21, 28, 35\) or \(42\) is spectral.  This is exactly what we will do.  

We also need the following symmetry result whose proof can be found in \cite{IosevichCrew}.

\begin{theorem}  \label{them:symmetry} 
Suppose \(E, A \subset \F_p^d\) and that \(\left( E, A \right)\) is a spectral pair.  Then \(\left( A, E \right)\) is a spectral pair.
\end{theorem}

\section{Auxiliary Lemmas in \(\F_p^3\)} \label{sectionconcentrate}
The first lemma and subsequent corollaries say that any non-tiling spectral set can not determine too many directions of any plane.  This means, in particular, that no such set can be too concentrated on any plane.  The second lemma proves analogous staments for lines.

\begin{lemma} \label{lemma:planedist}
Suppose \(E \subset \F_p^3\), \( \left| E \right| \leq p^2 \), and that there is some plane \(P_0\) such that \(\hat{E}\) vanishes on \(P_0 \setminus \left\{ 0 \right\}\). Then \( \left| E \right| = p^2 \).
\end{lemma}

\begin{proof}
Since \(\hat{E}\) has a zero, by Theorem \ref{equidist}, \(\left| E \right| = mp\), \(1 \leq m \leq p \).
Define \(d_0 := P_0^\perp\) and let \(\ell\) be an arbitrary translate of \(d_0\).  For \(1 \leq i \leq p+1\), let \(Q_i\) be the distinct planes containing \(\ell\).  Then, since \(Q_i^\perp \subset P_0\), \(\hat{E}\) vanishes on the nonzero points of \(Q_i^\perp\) for every \(i\),  we must have \(\left| E \cap Q_i \right| =m \). 

Notice that
\[
\F_p^3= \bigcup_{i=1}^{p+1}Q_i
\]
and for \(i \neq j\), \(Q_i \cap Q_j = \ell\).
Define \(a := \left| E \cap \ell \right|\) and \(b := \left| E \cap Q_i \right|-a \).  Observe that \(b\) is independent of \(i\).  Then
\begin{align*}
a+b&=m\\
a+\left(p+1 \right)b&=mp.
\end{align*}
\\
This implies \(m \equiv 0 \mod p\), which forces \(m=p\) and \( \left| E \right|=p^2\). This completes the proof. 
\end{proof}

\begin{corollary} \label{lemma:planedirect}
Suppose \(E, A \subset \F_p^3\) and that \(\left( E, A \right)\) is a spectral pair with \( \left| E \right| = \left| A \right| = mp \), \(2 \leq m \leq p-1\).  Then \(E\) and \(A\) each determine at most \(p\) directions of any plane.
\end{corollary}

\begin{proof}
Assume that \(A\) determines all \(p+1\) directions of some plane. That means that \(\hat{E}\) vanishes on every point of that plane except the origin.  By Lemma \ref{lemma:planedist}, \(\left| E \right| = p^2\), which is a contradiction.  By the symmetry property in Theorem \ref{them:symmetry}, \(E\) determining every direction of some plane also delivers a contradiction. This completes the proof that $A$ and $E$ each determine at most \(p\) directions of any plane.
\end{proof}

The following two results indicate a bound on a spectral set's the subspace concentration.

\begin{corollary} \label{lemma:planeconcentrate}
Suppose \(E, A \subset \F_p^3\) and that \(\left( E, A \right)\) is a spectral pair with \( \left| E \right| = \left| A \right| = mp \), \(2 \leq m \leq p-1\).  Then 
$$\sup_P \left| E \cap P \right| \leq p, \quad \text{ and }  \quad \sup_P \left| A \cap P \right| \leq p . $$ 
  Here, the supremum is taken over all possible  translates of all two dimensional subspaces.
\end{corollary}

\begin{proof} Assume that 
  there is a plane \(P_0\) with \( \left| E \cap P_0 \right| \geq p+1\). Then by pigeonholing, \(E\) determines every direction of some two dimensional subspace, which contradicts the previous corollary.  By symmetry, the same also holds for \(A\).
\end{proof}

\begin{lemma} \label{lemma:lineconcen}
Suppose \(E, A \subset \F_p^3\) and that \(\left( E, A \right)\) is a spectral pair with \( \left| E \right| = \left| A \right| = mp \), \(2 \leq m \leq p-1\).  Then
$$ 
\sup_{\ell} \left| E \cap \ell \right| \leq \min \left\{ m, p-m \right\}
 \quad \text{and} 
 \quad 
\sup_{\ell} \left| A \cap \ell \right| \leq \min \left\{ m, p-m \right\}.
$$
 Here, the supremum is taken over all possible translates of all one dimensional subspaces.
\end{lemma}
\begin{proof}
Fix a direction \(d\).  Since \(\left| A \right| > p\), \(A\) determines some direction of \(d^\perp\).  This means any translate of \(d\) is contained in a plane \(P\) with \(\left| E \cap P \right| = m\).  Since \(d\) is arbitrary,  then \(\sup_{\ell} \left| E \cap \ell \right| \leq m\).  By Theorem \ref{them:symmetry} the same holds for \(A\).

Let \(\ell\) be a traslate of \(d\).  Define \(\alpha := \left|E \cap \ell \right|\).  For \(1 \leq i \leq p+1\), let \(P_i\) be the distinct planes containing \( \ell\) and define \( \beta_i := \left|E \cap P_i \right| \). There is some \(i\) with \(\beta_i=m\), 
since \(A\) determines some direction of \(d^\perp\).  After a relabelling we may assume \(\beta_{p+1}=m\).  Observe that  by Corollary \ref{lemma:planeconcentrate} for all \(1 \leq i \leq p\) we have  \(\beta_i \leq p\) and     
\begin{align}\label{eq.1}
mp = \alpha + \left( m- \alpha \right) + \sum_{i=1}^p\left( \beta_i - \alpha \right) \leq m + p\left( p- \alpha \right).
\end{align}
By a rearrangement of the terms in (\ref{eq.1}) we obtain \( p \left(m+ \alpha -p \right) \leq m \).  Since \(p > m\) and everything in sight is an integer, it follows that
\[
m+\alpha \leq p.
\]
By the symmetry property, the same is also true for \(A\).
\end{proof}

As a corollary of the previous lemma, we rule out the spectral sets of size \(p(p-1)\).
\begin{corollary} \label{cor:nobigspectral} There is no spectral of size $p(p-1)$ in $\F_p^3\) . 
\end{corollary}

\begin{proof} Let   \(E\subset \F_p^3\) be a spectral of size \(p \left(p-1 \right)\). Then by 
  Lemma \ref{lemma:lineconcen} we must  have    \(\sup_\ell \left| E \cap \ell \right| \leq 1 \). This   implies \(E\) has at most one point, which is  a contradiction.
\end{proof}

The next result is an alternative and short proof for the Fuglede Conjecture in $\F_p^3$ for $p=3$. 

\begin{corollary}
The Fuglede Conjecture holds in \(\F_3^3\).
\end{corollary}

\begin{proof}
By Theorem \ref{IMPdimfree}, the only way this could fail is if there were a spectral set of size 6 in \(\F_3^3\), and by the Corollary \ref{cor:nobigspectral} this cannot happen. This proves  the claim. 
\end{proof}

\section{A Plancherel Calculation and the case \(p=5\)} \label{sectionplancherel}

In this section we study the structure of spectrums which determine almost as many directions of a plane as possible.  As a corollary of Lemma \ref{slabconcentration}, we prove the Fuglede's conjecture in \(\mathbb{F}_5^3\) in Corollary \ref{tenexclusion}. 

\begin{lemma}\label{LmNotSpectral}
Suppose \(E \subset \F_p^3\) with \(\left| E \right| = kp\), \(2 \leq k \leq p-1\).  If there is a plane \(P\) such that \(E\) is equidistributed on the planes parallel to \(d^\perp\) for at least \(p-1\) directions \(d \subset P\), then \(E\) is contained in the union of \(k\) parallel planes.
\end{lemma}

\begin{proof}
Let  \(d^\perp\) be a plane satisfying the hypotheses of the lemma.   If \(\ell\) is some translate of \(d\), a calculation similar to \eqref{eq.1} yields
\[
\left| \ell \cap E \right| + \left( p-1 \right)\left(k-\left| \ell \cap E \right| \right) +2 \left(p-\left| \ell \cap E \right| \right) \geq kp.
\]
Rearranging gives
\[
2p-k \geq p\left| \ell \cap E \right|,
\]
so \(E\) does not determine \(d\).  Thus projecting \(E\) by \(d\) gives a subset \(B \subset \F_p^2\) of size \(kp\) whose Fourier transform is supported on two lines through the origin.

For \(1 \leq i \leq p\),  let \(\left\{ L_i \right\}\) and \(\left\{ K_i \right\}\) be the two families of parallel lines on which \(B\) may not be equidistributed.  Pick this parametrization so that
\[
\bigcup_i \left( L_i \cap K_i \right)
\]
is a line.  Without loss of generality we can assume that  \(L_i \cap K_i \in B\) for  \(i=1,...,k\).  Finally, let \(L\) be the translate of \(L_1\) that contains the origin.  Similarly,  let \(K\) be the translate of \(K_1\) that contains the origin.    Define
\[
\bar{L}_i:= \left| L_i \cap B  \right|,  \ \ \bar{K}_i:= \left| K_i \cap B  \right|. 
\]
This labelling implies that
\begin{align} \label{KLcases}
 \bar{L}_i + \bar{K}_i = \begin{cases} 
  p+k \quad &\text{for} \ \  1 \leq i \leq k \\
k \quad &\text{for}   \ \ k+1 \leq i \leq p .
 \end{cases}
 \end{align}
 
 Identify   the  lines with their indicator functions and 
consider  the  following sum of squares of \(L^2\) norms
\[
I:= \sum_{x \in \F_p^2} \left| B \ast L \left( x \right) \right|^2 + \sum_{x \in \F_p^2} \left| B \ast K \left( x \right) \right|^2. 
\]
Using \eqref{KLcases}, the sum is equal to
\begin{align*}
I= p \sum_{i=1}^p\bar{L}_i^2 + \bar{K}_i^2 &= p \sum_{i=1}^p \left(\bar{L}_i + \bar{K}_i\right)^2 -  2\bar{L}_i\bar{K}_i \\ &= p \left( k \cdot (p+k)^2+ (p-k) \cdot k^2 - 2 \sum_{i=1}^p\bar{L}_i\bar{K}_i \right). 
\end{align*}
On the other hand, Plancherel identity (\ref{eqn:plancherel}) gives
\begin{align*}
I&= p^2 \sum_{x \in \F_p^2} \left|  \widehat{B \ast L} \left( x \right) \right|^2 + p^2\sum_{x \in \F_p^2} \left| \widehat{B \ast K} \left( x \right) \right|^2\\ 
& = p^4 \left( \sum_{x \in \F_p^2} |\hat{B}(x)|^2\right) + p^4|\hat{B}(0)|^2 \\
&= kp^3 + k^2p^2.
\end{align*}
Thus \(\sum_{i=1}^p\bar{L}_i\bar{K}_i = k^2 p\) and \(\hat{B}\) is in fact supported on a single line through the origin.
 

\end{proof}


Lemma \ref{LmNotSpectral} is particularly helpful for sets \(E\) where \(\sup_{\ell}|E \cap \ell|\) is small.
\begin{lemma} \label{slabconcentration}
Suppose \( \left( E, A \right) \) a spectral with \(\left|E\right|=kp\) and \(E\) is contained in the union of \(k\) parallel planes.  If  \(\sup_\ell \left| E \cap \ell \right| = 2\), then \(k\) is even and \(A\) determines at most \(k-1\) directions of \(p\) planes, all of which intersect at the line through the origin perpendicular to the given planes.
\end{lemma}

\begin{proof}
Let \(P_1, \ldots, P_k\) be the given parallel planes and let \(P\) be the translate of \(P_1\) that contains the origin.  Note \(E\) does not determine every direction of \(P\).  Also note that no direction of \(E \cap P_i\) is determined more than \(\frac{p-1}{2}\) times.  Since there are \(\frac{p \left( p-1 \right)}{2}\) pairs of points of \(E \cap P_i\), each direction of \(P\) except one is determined \(\frac{k\left( p-1 \right)}{2}\) times; indeed,\(\frac{\left( p-1 \right)}{2}\) from each plane.

If \(k\) is odd, and \(\ell^*\) is any direction of \(P\) determined by \(E\), the inequality
\[
k \left(p-1\right) > \left(k-1\right)p
\]
implies that \(E\) is not equidistributed on any family of parallel planes containing \(\ell^*\), which contradicts spectrality. Therefore, $k$ must be  even.   Further, \(E\) can not be equidistributed on a family of parallel planes containing \(\ell^*\) unless each of the planes contains an even number of points that do not determine \(\ell^*\).  So any equidistribution induces a (unique, non-overlapping) pairing of one of these points with at least one other.  There are at most \(k-1\) such parings.
\end{proof}

\begin{corollary} \label{p2exclusion}
No subset of size $p(p-2)$ is spectral. 
\end{corollary}

\begin{proof}
As before, we will assume that such a spectral set exists and derive a contradiction.  Let 
\(E \subset \F_p^3\) and \(\left| E \right| = p\left( p-2 \right) \).  Assume that  \(E\) has a spectrum \(A\).  Define \(m:= \sup_P \left| A \cap P\right|\).  There is some line \(\ell\) with \(\left| A \cap \ell \right|=2\). It follows that
\[
2 + \left(p+1 \right) \left(m-2 \right) \geq \left(p-2 \right)p, 
\]
which implies
\[
m \geq p-1+\frac{1}{p+1}.
\]
Since \(m\) is an integer,  we have \(m \geq p\).  By Corollary \ref{lemma:planeconcentrate}, we must then have \(m=p\).  By Lemma \ref{LmNotSpectral}, \(E\) is contained in the union of \(p-2\) parallel planes.  Since no three points of \(E\) are collinear, \(p-2\) is even by Lemma \ref{slabconcentration}, which is a contradiction.
\end{proof}

All results up to this point hold for all \(p\). We now address the cases of \(p=5,7\).

\begin{corollary} \label{tenexclusion}
The Fuglede conjecture holds in \(\F_5^3\).
\end{corollary}

\begin{proof}
From the previous corollary, we now only need to consider sets of size \(10\).  Suppose \(\left(E,A\right)\) is a spectral pair with \(\left| E \right| = \left| A \right| = 10\).  By a similar calculation to \eqref{eq.1}, if \(A\) determines \(d_0\) there is a plane \(P_0\) containing a translate of \(d\) with \(\left|A \cap P_0 \right| \geq 4\).  By Lemma \ref{LmNotSpectral}, Lemma \ref{slabconcentration} and the symmetry condition,  \(E\) determines at most ten directions and is contained in the union of two parallel planes.  This is a contradiction, since no subset of \(\F_5^2\) is closed under a nontrivial translation.
\end{proof}

\section{The case \(p=7\)} \label{sectionseven}

The strategy for \(p=7\) is similar in spirit to the \(p=5\) case, but more complicated. Recall that by Theorem \ref{IMPdimfree}, a nontrivial spectral set \(E\) has size \(|E| = 7m\) for \(1 \leq m \leq 7\).  In this section, we rule out the possibility of spectral sets of size \(14, 21\), and \(28\).  We begin with a two dimensional lemma.

\begin{lemma}\label{Lm1}
Suppose \(E \subset \F_7^2\) has five points, no three of which are collinear.  Then \(E\) determines at least six directions.
\end{lemma}

\begin{proof}
The collinearity restriction means that any direction is realized at most twice.  Since \(\binom{5}{2}= 10\), \(E\) must determine five distinct directions. 
For the sake of contradiction, suppose  that the set \(E\) determines exactly 5 directions.  Then each direction   is realized exactly twice.

If some 4 point subset \(E' \subset E\) determines only 4 directions, that implies that the remaining direction is realized by two lines joined to the same point, which contradicts the collinearity assumption.

So, without loss of generality, assume that 
\[
E = \left\{ \left(0, 0 \right), \left(0, 1 \right), \left(1, 0 \right), \left(1, a \right), \left(b, 1 \right) \right\}, 
\]

\noindent where \(a \neq 1, 6\) and \(b \neq 1,6 \).  This means that \(\left(1, a \right)\) is a multiple of \( \left(b-1, 1 \right)\) and \(\left(b, 1 \right)\) is a multiple of \(\left( 1, a-1 \right) \), which implies
\[
b^2-b-1=0
\]
which has no solution in \(\F_7\), and we are done. 
\end{proof}

We can now show that spectral sets of size 28 and 14 do not exist.  The size 21 case will require further lemmas.

\begin{corollary}\label{28} 
There is no spectral set of size 28 in \(\F_7^3\). 
\end{corollary}

\begin{proof}
If \(E \subset \F_7^3\) is spectral with size 28 and spectrum \(A\), then \(\sup_\ell \left|E \cap \ell \right| \leq 3\), and so \(E\) determines at least four directions of any plane.  If, in fact, \(\sup_\ell \left|E \cap \ell \right| = 3\), the bound \(\sup_P \left|E \cap P \right| \leq 7\) gives the calculation
\[
23 = 1+1+1+1+3+4+4+4+4 \geq 28,
\]
which is false.  So \(\sup_\ell \left|E \cap \ell \right| = \sup_\ell \left|A \cap \ell \right| = 2\).
Then \(\sup_P \left|A \cap P \right| \geq 5\) implies, by Lemmas \ref{Lm1} and   \ref{slabconcentration} that \(A\) determines at most three directions of a plane, which cannot be.
\end{proof}

\begin{corollary}\label{14}
There is no spectral set of size 14 in  \(\F_7^3\). 
\end{corollary}

\begin{proof}
As before, let \(\left(E, A \right)\) is a spectral pair with \( \left| E \right| = \left| A \right|=14\).  If we have \(\sup_P \left|A \cap P \right| \geq 5\), then by Lemma \ref{Lm1}, Lemma \ref{slabconcentration}, and symmetry property Theorem \ref{them:symmetry}, we have that \(E\) is contained in the union of two parallel planes and determines at most 14 directions.  This is a contradiction, since no non-empty subset of \(\F_7^2\) is its own non-trivial translate.

So we have
\begin{equation} \label{ineqaulfourt}
\sup_P \left| A \cap P \right| \leq 4 \,\, \text{and} \, \, \sup_P \left|E \cap P \right| \leq 4.
\end{equation}
Let \(E\) determine \(d_0\). The inequality \eqref{ineqaulfourt} implies that \(E\) is equidistributed on at most two families of parallel planes containing \(d_0\).  So \(A\) determines at most two directions of \(d_0^\perp\), call them \(d_1\) and \(d_2\).  Without loss of generality, \(A\) determines \(d_1\) more often than  \(d_2\).

The claim is that for all directions \(d \subset d_1^\perp\) except \(d_0\), \(A\) is not equidistributed on \(d^\perp\).

If \(A\) determines \(d_1\) seven times, projecting by \(d_1\) gives a 7 element subset of \(\F_7^2\) with no three points collinear.  Such a subset must determine 7 directions, which proves the claim.

Now we are left to consider the case where \(d_1\) is determined four, five, or six times.  Observe that the existence of a plane \(P\) such that \(P \cap E\) determines \(d_1\) and \(\left| P \cap E \right| = 3\) implies the claim.  If \(d_1\) is determined five or six times, such a plane must exist.  If \(d_1\) is determined four times, the non-existence of such a plane also implies the claim.

Then, by symmetry, \(E\) determines \(d_0\) seven times.  Since \(d_0\) was arbitrary, any direction determined by \(E\) is determined seven times.  This implies \(\sup_P \left| E \cap P \right|=2 \), which is false.
\end{proof}

\subsection{Proof of $E$ with size $|E|=21$ is not spectral}\label{section21notspectral}
Finally, we give the last three lemmas which we  need to take care of the size 21 case.

\begin{lemma}\label{Lm2}
Suppose \(E \subset \F_7^2\) has seven points, no four of which are collinear.  Then \(E\) determines at least six directions.
\end{lemma}

\begin{proof}
Suppose, for the sake of contradiction, that \(E\) determines five or fewer directions.  Since \(\binom{7}{2} = 21\), some direction is determined at least 5 times.  By the Cauchy-Davenport Theorem \cite{TaoVu}, there can be no distinct two parallel lines that each contain three or more points.  And so, by the equality case of Cauchy-Davenport \cite{Vosperone, Vospertwo}, without loss of generality we assume that 
\[
E = \left\{ \left(0, 0 \right), \left(0, 1 \right), \left(0, 2 \right), \left(1, 0 \right), \left(1, 1 \right), \left(c, d \right), \left(c, d+1 \right) \right\}
\]
where \(c \neq 0,1\) and \(d \neq 0\).  Matching directions  implies 
\begin{align*} \left\{ \left(1, -2 \right), \left(1, -1 \right), \left(1, 0 \right), \left(1, 1 \right) \right\} =  \left\{ \left(1, \cfrac{d-2}{c} \right),\left(1, \cfrac{d-1}{c} \right), \left(1, \cfrac{d}{c} \right), \left(1, \cfrac{d+1}{c} \right)\right\}.
\end{align*}

\noindent This  then gives
\[
\Omega := \left\{-2c, - c, 0, c \right\} = \left\{d-2, d-1, d, d+1 \right\}.
\]

Since \(0 \in \Omega\), either \(1 \in \Omega\) or \(-1 \in \Omega\) (or both).  This means \(c= 1,3,4\) or \(-1\). Also, \(\left\{3, 4 \right\}\) is not a subset of \(\Omega\).  This gives \(c=-1\) and \(d=1\), which then implies \(E\) determines at least six directions.
\end{proof}

\begin{lemma} \label{twoparalellines}
Suppose \(E \subset \F_7^3\), \(\left| E \right|=21\), and that \(E\) is spectral.  If \(E\) determines no more than two directions of two distinct planes, then \(E\) is contained in the union of seven parallel lines.
\end{lemma}

\begin{proof}
Call the two planes \(P_1\) and \(P_2\).  By assumption there are directions \(\left\{d_1, d_2 \right\} \subset P_1\) and \(\left\{ e_1, e_2 \right\} \subset P_2\) such that for each \(x \in E\), there is a translate \(\ell_x\) of either \(d_1\) or \(d_2\) with \(x \in \ell_x\) and \( \left| E \cap \ell_x \right|=3\).  Similarly, there is a translate \(k_x\) of either \(e_1\) or \(e_2\) with \(x \in k_x\) and \( \left| E \cap k_x \right|=3\).


Let \(D \subset \left\{ d_1, d_2, e_1, e_2 \right\}\) be such that each \(d \in D\) is determined by \(E\).  Note that \(D\) must contain at least one of \(\{d_1, d_2\}\) and at least one of \(\{e_1, e_2\}\).  If one of these is missing, without loss of generality suppose \(d_2\) is missing, then \(E\) is contained in seven translates of \(d_1\), which proves the claim.

If \(\left| D \right| =4\), fix \(x_0 \in E\) and let \(x_1, x_2\) be the two other points of \(E\) with \(x_1 \in \ell_{x_0}\) and \(x_2 \in \ell_{x_0}\).  After possibly relabeling, we may assume that \(k_{x_1}\) and \(k_{x_2}\) are parallel.  Let \(y_1\), \(y_2\) be the two other points of \(k_{x_1}\), and \(z_1\), \(z_2\) the two other points of \(k_{x_2}\).  Then, since no plane contains eight or more points, \(\ell_{y_1}, \ell_{y_2}, \ell_{z_1}, \ell_{z_2}\) are all parallel.  By symmetry, this means both \(d_1\) and \(d_2\) are determined 12 or more times, which is a contradiction.
\end{proof}


\begin{lemma} \label{twentyoneproj}
Suppose \(f: \F_7^2 \to \left\{0, 1, 2, 3 \right\}\) is supported on three parallel lines \(\ell_1,\ell_2, \ell_3\), 
that for any line \(\ell\),
\begin{equation} \label{eq:linebound}
\sum_{x \in \ell} f\left(x \right) \leq 7,
\end{equation}
and for the lines \(\ell_1,\ell_2, \ell_3\), 
\[
\sum_{x \in \ell_1} f\left(x \right) = \sum_{x \in \ell_2} f\left(x \right)= \sum_{x \in \ell_3} f\left(x \right)=7.
\]

Then, if \(f\left(x \right)=3\) for any \(x \in \F_7^2\), \(f\) is equidistributed on at most two families of parallel lines.
\end{lemma}

\begin{proof}
Suppose \(f\) is equidistributed on three of more families of parallel lines.  Then, by Cauchy-Davenport, for any \(i \neq j\), there are at most five points \(x \in \ell_i \cup \ell_j\) with \(f \left( x \right) \geq 2\) and \(f\) is supported on at most \(4\) points of each \(\ell_i\)

If \(f\left( x \right) = f\left( y \right) =3\) for any two distinct points of \(\ell_1\), then \(f\) is supported on exactly three points of \(\ell_2\) and exactly three points of \(\ell_3\), again by Cauchy-Davenport.  This implies that
\[
\left| f^{-1} \left( 1 \right) \right| = 3, \left| f^{-1} \left( 2 \right) \right| = 0 \,\, \text{or} \left| f^{-1} \left( 1 \right) \right| = 2, \left| f^{-1} \left( 2 \right) \right| = 2.
\]In either case, one arrives at a contradiction.

Without loss of generality, we are reduced to the case where, in multiset notation,
\[
\left\{f\left(x \right) \mid x \in \ell_1 \right\} = \left\{f\left(x \right) \mid x \in \ell_2 \right\} = \left\{0, 0, 0, 1, 1, 2, 3 \right\}.
\]

If
\[
\left\{f\left(x \right) \mid x \in \ell_3 \right\} = \left\{0, 0, 0, 1, 2, 2, 2 \right\} \,\,\text{or} \, =\left\{0, 0, 0, 0, 2, 2, 3 \right\} ,
\]
then \(f\) equidistributes on at most two families of parallel lines.  This can be seen by matching \(1\)'s and \(2\)'s.

If \(f\left( \ell_3 \right) = f \left( \ell_1 \right)\) as multisets, then \(T:= f^{-1} \left\{2, 3 \right\}\) satisfies \(\left| T \right| = 6\).  We will show \(T\) determines at least six directions.  If there is a five element subset of \(T\) with no three points collinear, the claim follows by Lemma \ref{Lm1}.  If not, then \(T\) is contained in the union of three lines.  Since by \eqref{eq:linebound} none  of these three lines are parallel, so  \(T\) determines at least four directions at least three times.  The pairs of points remaining are the two element subsets of \(f^{-1} \left( 3 \right)\).  Notice that \eqref{eq:linebound} implies that \(f^{-1} \left( 3 \right)\) determines three directions.  Therefore, if \(T\) determines no direction four or more times, the claim follows.

What remains is the situation where \(T\) determines some direction at least four times.  In this case, there is a five element subset of \(T\) contained in the union of two parallel lines.  If the three and two element subsets obtained by intersecting \(T\) with each of these lines are not arithmetic progressions of the the same difference, then the claim follows by the equality case of Cauchy-Davenport.  If they are, then without loss of generality,
\[
T' := \left\{ \left(0, 0 \right), \left(0,1 \right), \left(0, 2 \right), \left( 1,0 \right), \left( 1,1 \right) \right\}
\]
is a subset of \(T\).  There is no point \(p \in \F_7^2\) on the line
\[
  \left\{ \left(z, 2 \right) \mid z \in \F_7 \setminus \left\{0 \right\} \right\}
\]
such that \(\left\{ p \right\} \cup T' \) determines five or fewer directions.  This completes the proof.
\end{proof}

\begin{corollary}\label{21}
There is no spectral set of size 21 in  \(\F_7^3\). 
\end{corollary}

\begin{proof}
Assume that  \(E \subset \F_7^3\) has size 21 and has spectrum \(A\).  If \(A\) determines six directions of a plane,  then by Lemma \ref{LmNotSpectral}, there is some plane \(Q_0\) such that \(E\) is contained in three disjoint translates of \(Q_0\).

If \(d \subset Q_0\) is determined, then \(E\) is equidistributed on at most four families of parallel planes containing \(d\).  To see this, project by \(d\) to obtain a function \(g\) on \(\F_7^2\) satisfying the hypotheses of Lemma \ref{twentyoneproj} except possible without an \(x\) such that \(g \left( x\right) =3\).  If there is such an \(x\), we are done by Lemma \ref{twentyoneproj}.

So assume \(g: \F_7^2 \to \left\{0, 1, 2\right\}\) and is supported on the parallel lines \(\ell_1, \ell_2, \ell_3\).  Since \(d\) is determined, there is a \(y \in \ell_1\) with \(g \left( y\right) =2\).  If \(g\) were equidistributed on five or more families of parallel lines, there would be points
\[
\left\{w_1, w_2, w_3, w_4, w_5\right\} \subset \ell_2 \cup \ell_3
\]
with \(g \left( y \right)=0\).  Since each such \(\ell_i\) contains at most three zeros, by Cauchy-Davenport, \(g\) is not equidistributed on at least 5 families of parallel lines.  This is a contradiction.

Note that \(E\) determines no direction of \(Q_0\) more than 18 times.  Also, observe that if \(d \subset Q_0\) is determined by \(E\) and there are three or more families of parallel planes containing \(d\) on which \(E\) equidistributes, \(d\) is determined six or fewer times.  Indeed, any direction determined 7 or more times must either have three points on a line or determine the direction on seven distinct parallel lines.  In the former case, we are done by Lemma \ref{twentyoneproj}, in the latter by Lemma \ref{Lm2}.

Let \(s\) be the number of directions \(d \subset Q_0\) for which there are two or fewer families of parallel planes containing \(d\) on which \(E\) equidistributes.  If \(E\) determines six directions of \(Q_0\) then
\[
6 \left(6 - s\right) + 18s \geq 3 \binom{7}{2} = 63,
\]
and so \(s \geq 3\).  If \(E\) determines seven directions of \(Q_0\), then
\[
6 \left(7 - s\right) + 18s \geq 63,
\]
and so \(s \geq 2\).  In either case, by Lemma \ref{twoparalellines}, \(E\) is contained in the union of seven parallel lines, and by Cauchy-Davenport, determines at least 36 directions.

On the other hand, by the previous discussion, if \(E\) determines seven directions of \(Q_0\), it is equidistributed on at most
\[
7\left( 1\right) + 2\left( 2\right) + 4\left( 5\right) =31
\]
families of parallel planes.  If \(E\) determines six directions of \(Q_0\), it is equidistributed on at most
\[
7\left( 2\right) + 2\left( 3\right) + 4\left( 3\right) =32
\]
families of parallel planes.  By symmetry, the same is true for \(A\), and so we reach a contradiction.

Therefore, neither \(E\) nor \(A\) determine more than five directions of any plane.  In particular,
\[
\sup_P \left| E \cap P \right| \leq 6 \,\, \text{and} \,\, \sup_P \left| A \cap P \right| \leq 6.
\]

If \(\left| E \cap \ell \right|=3\) for some line \(\ell\), then \(\sup_P \left| E \cap P \right| \leq 6\) implies \(E\) is equidistributed on at most two families of parallel planes containing \(d\).  So if \(A\) is a spectrum of \(E\), then \(A\) determines at most two directions of \(d^\perp\).  This means that there are two parallel lines \(\ell_1\) and \(\ell_2\) such that \(\left| A \cap \ell_1 \right| = \left| A \cap \ell_2 \right|=3\).  By Cauchy-Davenport, \(A\) determines at least six directions of a plane, which can not be possible.

So \(\sup_\ell \left| E \cap \ell \right| = 2\).  This means there's a plane \(P\) with \(\left| E \cap P \right| \geq 5\), which means by Lemma \ref{Lm1}, \(E\) determines at least six directions of a plane, which again, can not happen.   So we must have \(\sup_\ell \left| E \cap \ell \right| = 1\), which is another contradiction.  This completes the proof of the corollary, as well as the main theorem \ref{thrm:main}.
\end{proof}

\bibliography{Final2.bbl}
\bibliographystyle{amsalpha}

\end{document}